\documentclass[12pt, a4paper]{amsart}

\usepackage[english]{babel}
\usepackage{amsthm}
\usepackage{amssymb}
\usepackage{amsxtra}
\usepackage{amscd}
\usepackage[pdftex]{graphicx}
\usepackage[matrix,arrow,curve]{xy}
\usepackage{euscript}

\usepackage[pdftex]{hyperref}

\theoremstyle{plain}
\newtheorem{theoremX}{Theorem}
\newtheorem{theorem}{Theorem}[section]
\newtheorem{lemma}[theorem]{Lemma}

\theoremstyle{definition}
\newtheorem{definition}[theorem]{Definition}

\newtheorem{example}[theorem]{Example}

\newcommand{\pairing}[2]{\left\langle {#1},{#2}\right\rangle}
\renewcommand{\dim}{\mathrm{dim}}
\newcommand{\Aut}{\mathrm{Aut}}

\newcommand{\cone}{\mathrm{cone}_{\mathbb Q}}
\newcommand{\D}{\Delta}
\newcommand{\R}{\EuScript R}

\renewcommand{\P}{{\mathbb P}}

\newcommand{\Cl}{\mathrm{Cl}}

\begin{document}

\title[Automorphisms of a toric variety]{On orbits of the automorphism group on \\ a complete toric variety} 
\author{Ivan Bazhov}
\address{Algebra Department, Faculty of Mechanics and Mathematics,\linebreak
Moscow State University, Leninskie Gory 1,
Moscow, 119991, Russia}
\email{ibazhov@gmail.com}
\date{}

\begin{abstract}
Let $X$ be a complete toric variety and $\Aut(X)$ be the automorphism group. We give an explit description of $\Aut(X)$-orbits on~$X$. In particular, we show that $\Aut(X)$ acts on $X$ transitively if and only if $X$ is a product of projective spaces.
\end{abstract}

\maketitle{}

\section{Introduction}

Let us recall that a toric variety is a normal algebraic variety $X$ with a generically transitive and effective action of an algebraic torus~$\mathbb T$. The complement $X\smallsetminus\mathcal O_0$ of the open $\mathbb T$-orbit $\mathcal O_0$ is a union of prime $\mathbb T$-invariant divisors $D_1,\ldots,D_r$. Every  $\mathbb T$-orbit on $X$ is uniquely defined by the divisors~$D_i$ which contain the orbit. In particular, the number of $\mathbb T$-orbits is finite. 

It is known that the group of algebraic automorphisms $\Aut(X)$ of a projective or, more generally, complete toric variety $X$ is an affine algebraic group. The group $\Aut(X)$ contains $\mathbb T$ as a maximal torus. Hence, the connected component of identity $\Aut^0(X)$ is generated by the torus $\mathbb T$ and root subgroups, i.e., one dimensional unipotent subgroups normalized by $\mathbb T$. For smooth $X$ a combinatorial description of root subgroups was found by M.\,Demazure \cite{Dem1}. A generalization of this description to complete simplicial toric varieties in terms of total coordinates was given in~\cite{Cox2}. The case of an arbitrary complete toric variety was studied in \cite{Buh}, see also \cite{Nill}. The paper \cite{BG} worked out another method to describe the automorphism group of a projective toric variety.

It is clear that any $\Aut^0(X)$- as well as any $\Aut(X)$-orbit on $X$ is a union of $\mathbb T$-orbits. Our goal is to describe  $\mathbb T$-orbits which lie in the same $\Aut^0(X)$- or $\Aut(X)$-orbit. 

Let us associate with a $\mathbb T$-orbit $\mathcal O$ the set of all prime $\mathbb T$-invariant divisors $D(\mathcal O):=\{D_{i_1},\ldots,D_{i_l}\}$ which do not contain $\mathcal O$. Let $\Gamma(\mathcal O)$ be the monoid in the divisor class group $\Cl(X)$ generated by classes of the divisors in~$D(\mathcal O)$.

\begin{theoremX}
\label{Uno}
Let $X$ be a complete toric variety. Then two $\mathbb T$-orbits $\mathcal O$ and $\mathcal O'$ on $X$ lie in the same $\Aut^0(X)$-orbit if and only if 
$$
\Gamma(\mathcal O)=\Gamma(\mathcal O').
$$
\end{theoremX}	

The proof uses techniques of paper \cite{AKZ}. It is shown there that for an affine toric variety any non-trivial orbit of a root subgroup meets exactly two $\mathbb T$-orbits and a combinatorial description of such pairs of orbits is given. We adapt these results to a complete toric variety.

Paper \cite{BH} gives a systematic treatment of toric varieties in term of bunches of cones. These cones are images of $\Gamma(\mathcal O)$ in $\Cl(X)\otimes\mathbb Q$ and they are Gale dual to cones in $\D$.

Theorem \ref{New_Th} brings a description of the closures of $\Aut^0(X)$-orbits on $X$. Theorem \ref{Uno_ii} gives a description of $\Aut(X)$-orbits. As a corollary we show that the action of $\Aut(X)$ on $X$ is transitive if and only if $X$ is a product of projective spaces. A related result obtained in \cite[Proposition 4.1]{AG} states that every complete toric variety which admits a transitive action of a semisimple algebraic group is a product of projective spaces.

The base field $\Bbbk$ is assumed to be algebraically closed and of characteristic zero.

The author is grateful to his supervisor I.\,V.\,Arzhantsev for stating the problem, the idea to apply Gale duality and useful references, and to the referee for his/her careful reading of the manuscript and many helpful comments.

\section{Preliminaries}
\label{intro}

{\bf Toric varieties and fans.}
Let $N\cong\mathbb Z^d$ be the lattice of one-parameter subgroups of a torus $\mathbb T$,  $M$ be the dual lattice of characters, and $\pairing{\cdot}{\cdot}:M\times N\to \mathbb Z$ be the natural pairing. Also we put $N_{\mathbb Q}=N\otimes_{\mathbb Z}\mathbb Q$, $M_{\mathbb Q}=M\otimes_{\mathbb Z}\mathbb Q$. 

With any toric variety $X$ one associates a fan $\D=\D_X$ of rational polyhedral cones in $N_{\mathbb Q}$, e.g., see \cite{Fulton, Cox1, Oda1}. Let $\D(1)=\{\tau_1,\ldots,\tau_r\}$ be the set of one-dimensional cones in $\D$ and $\sigma(1)$ be the set of one-dimensional faces of a cone $\sigma$. The primitive lattice generator of a ray $\tau$ is denoted as $v_\tau$ and $v_i=v_{\tau_i}$.

There exists a one-to-one correspondence $\sigma\leftrightarrow \mathcal O_\sigma$ between cones $\sigma$ in $\D$ and $\mathbb T$-orbits $\mathcal O_\sigma$ on $X$ such that $\dim\,\mathcal O_\sigma=\dim\,\D-\dim\,\sigma$. A toric variety $X$ is complete if and only if the fan $\D$ is complete.

Every prime $\mathbb T$-invariant divisor $D_i$ is a closure of an orbit $\mathcal O_{\tau_i}$ for some $\tau_i\in\D(1)$. A $\mathbb T$-orbit $\mathcal O_\sigma$ is contained in the intersection of divisors $D_i$ corresponding to rays $\tau_i\in\sigma(1)$. 

There is a natural map from the free abelian group $\mathrm{Div_{\mathbb T}}(X)$ generated by all prime $\mathbb T$-invariant divisors to the divisor class group $\Cl(X)$: $D\mapsto [D]$. If $X$ is complete then the following sequence is exact \cite[Theorem 4.1.3]{Cox1}:
\begin{equation}
\label{eq_exact}
0\to M\to \mathrm{Div_{\mathbb T}}(X)\to \Cl(X)\to 0,
\end{equation}
where $m\in M$ goes to $\sum\limits_{i=1}^{r}\pairing{m}{v_i}D_i$.

\begin{example}
\label{example_1}
Let us give some examples of complete toric surfaces.
\begin{enumerate}
	\item The projective plane $\P^2$, the fan $\D_{\P^2}$ is shown in Figure~\ref{fig_1}(i),\\
	$\D_{\P^2}(1)$ is generated by $(1,0)$, $(0,-1)$, $(-1,1)$, $\Cl(\P^2)\cong\mathbb Z$.
	\item The Hirtzebruch surface $\EuScript H_s$, its fan is shown in Figure~\ref{fig_1}(ii),\\
	$\D_{\EuScript H_s}(1)$, $s\geqslant1$, is generated by $(1,0),\ (0,-1)$, $(-1,s)$, $(0,1)$,\\ $\Cl(\mathcal H_s)\cong\mathbb Z^2$.
	\item The weighted projective plane $\P(1,1,s)$, $s\geqslant2$, is defined by the fan in Figure~\ref{fig_1}(iii),
	$\D_{\P(1,1,s)}(1)$ is generated by $(1,0)$, $(0,-1)$, $(-1,s)$, $\Cl(\P(1,1,s))\cong\mathbb Z$.
	\item The variety $\EuScript B_s$, $s\geqslant1$, is defined by the fan in Figure~\ref{fig_1}(iv),\\
	$\D_{\EuScript B_s}(1)$ is generated by  $(s,1)$, $(s,-1)$, $(-s,-1)$, $(-s,1)$,\\ $\Cl(\EuScript B_s)\cong\mathbb Z^2\oplus \mathbb Z_{2s}$.
\end{enumerate}
\begin{figure}[htbp]
	\centering
		\includegraphics[width=120mm]{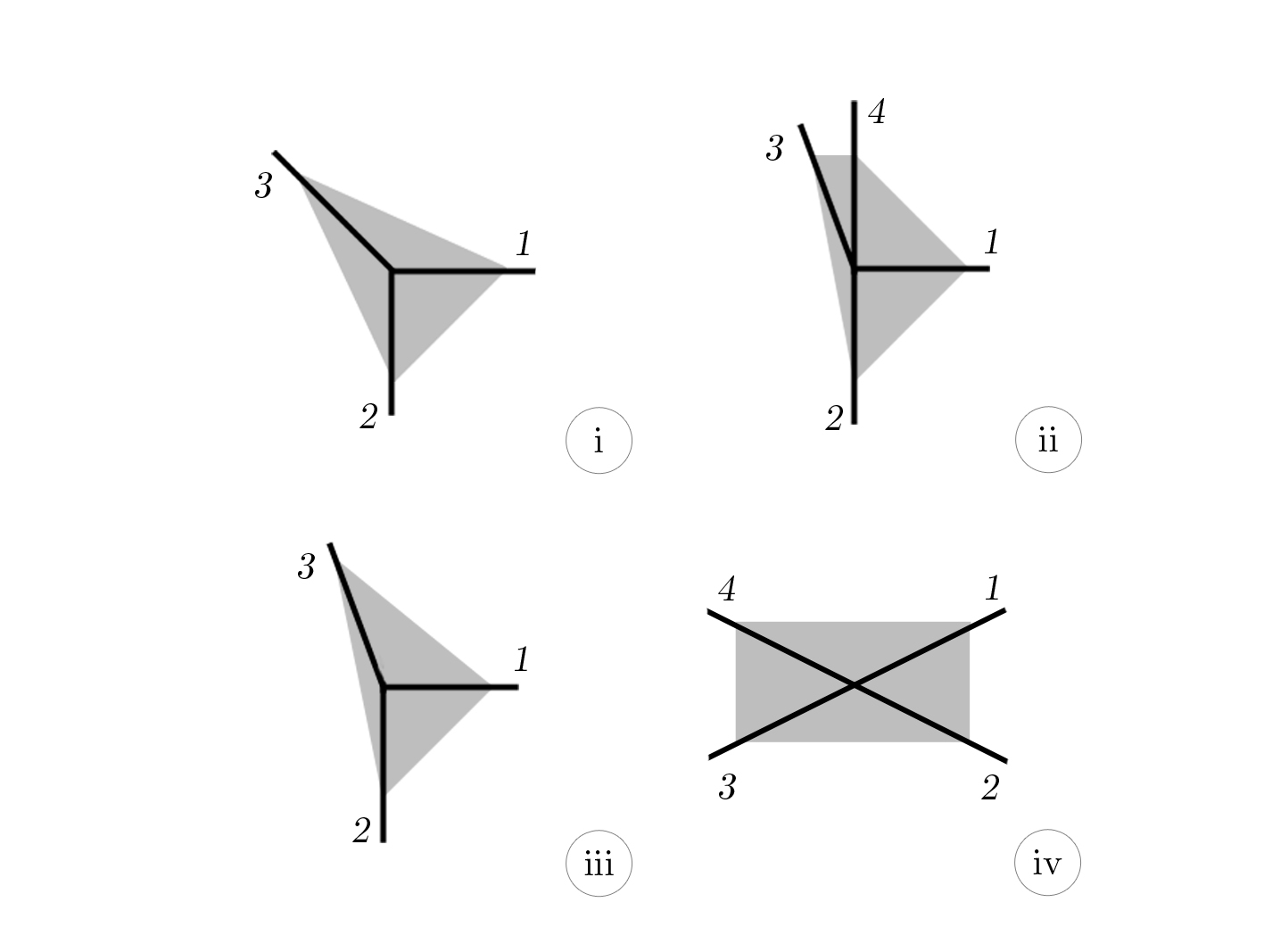}
	\caption{Fans of $\P^2$, $\EuScript H_s$, $\P(1,1,s)$, and $\EuScript B_s$.}
	\label{fig_1}
\end{figure}
\end{example}

{\bf Root subgroups and the automorphism group.} Let $X$ be a complete toric variety and $\D$ be the corresponding fan.
\begin{definition}
\label{root_def}
An element $m\in M$ is called {\it a Demazure root} of $\D$ if the following conditions hold:  
$$
\exists \tau\in\D(1): \pairing{m}{v_\tau}=-1\mbox{ and } \forall \tau'\in\D(1)\smallsetminus\{\tau\}: \pairing{m}{v_{\tau'}}\geqslant0.
$$
\end{definition}
Let us denote the set of all roots of the fan $\D$ as $\R$. Denote by $\eta_m$, $m\in\R$, the unique ray generator $v_\tau$ such that $\pairing{m}{v_\tau}=-1$. A root $m\in\R\cap(-\R)$ is called {\it semisimple}.

\begin{example}
\label{example_2}
The sets of roots for the varieties of Example~\ref{example_1} are easy to picture.
\begin{enumerate}
	\item Roots for $\P^2$ are pictured in Figure~\ref{fig_2}(i). 
	\item Roots for $\EuScript H_s$ are pictured in Figure~\ref{fig_2}(ii).
	\item The set of roots for $\P(1,1,s)$ and for $\EuScript H_s$ coincide.
	\item The set of roots for $\EuScript B_s$ is empty.
\end{enumerate}
\begin{figure}[htbp]
	\centering
		\includegraphics[width=120mm]{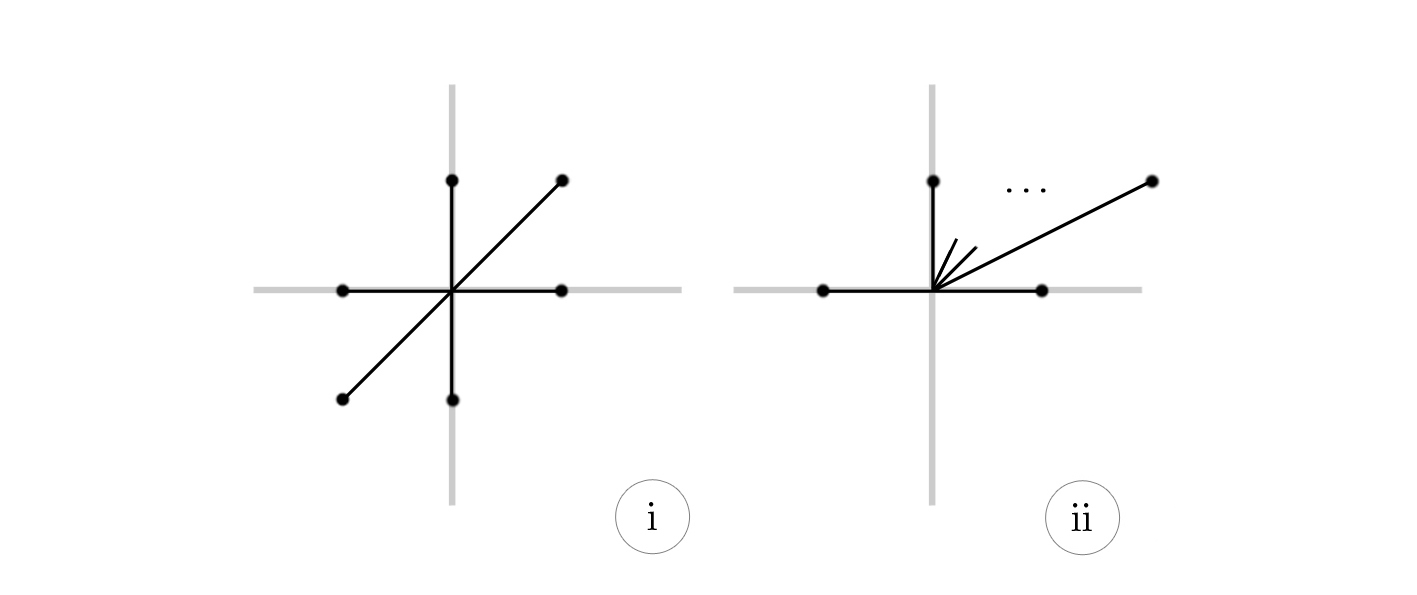}
	\caption{Demazure roots for $\P^2$, $\EuScript H_s$, and $\P(1,1,s)$.}
	\label{fig_2}
\end{figure}
\end{example}

A root $m\in\R$ defines a root subgroup $H_m\subset\Aut(X)$, see \cite{AKZ, Cox2, Oda1}. The root subgroups $H_m$ are unipotent one-parameter subgroups normalized by $\mathbb T$.

\begin{lemma}
\label{H_mconn}
For every point $Q\in X\smallsetminus X^{H_m}$ the orbit $H_m\cdot Q$ meets exactly two $\mathbb T$-orbits $\mathcal O_1$ and $\mathcal O_2$ on $X$ and $\dim\,\mathcal O_{\sigma_1}=1+\dim\,\mathcal O_{\sigma_2}$.
\end{lemma}

\begin{proof}
Let $U_i$, $i\in I$, be affine charts corresponding to cones in $\D$ which contain $\eta_m$. Due to the proof of \cite[Proposition 3.14]{Oda1}, the charts $U_i$ are preserved by $H_m$ and the complement $V=X\smallsetminus \bigcup\limits_{i\in I}U_i$ is fixed by $H_m$ pointwise. Applying \cite[Proposition 2.1]{AKZ} to affine charts $U_i$ completes the proof.
\end{proof}

\begin{definition}
\label{H_m_def}
A pair of $\mathbb T$-orbits $(\mathcal O_{\sigma_1}, \mathcal O_{\sigma_2})$ from Lemma \ref{H_mconn} is called {\it $H_m$-connected}. 
\end{definition}

\begin{lemma}(\cite[Lemma 2.2]{AKZ})
\label{H_m_lemma}
A pair $(\mathcal O_{\sigma_1}, \mathcal O_{\sigma_2})$ is $H_m$-connected if and only if 
\begin{equation}
\label{pair}
m|_{\sigma_2}\leqslant 0\mbox{ and } \sigma_1=\sigma_2\cap m^{\bot}\mbox{ is a facet of cone }\sigma_2.
\end{equation}
\qed
\end{lemma}

Every automorphism $\psi\in\Aut(N,\D)$ of the lattice $N$ preserving the fan $\D$ defines an automophism of $X$ denoted by $P_\psi$.

\begin{theorem}(\cite[Corollary 4.7]{Cox2}, \cite{Dem1}, \cite[page 140]{Oda1})
\label{all_aut}
Let $X$ be a complete toric variety. Then $\Aut(X)$ is a linear algebraic group generated by $\mathbb T$,
$H_m$ for $m\in\R$, and $P_\psi$ for $\psi\in\Aut (N,\D)$. Moreover, the group $\Aut^0(X)$ is generated by $\mathbb T$ and
$H_m$ for $m\in\R$.\qed
\end{theorem}


\section{Main results}
\label{main}

Let $X$ be a complete toric variety defined by the fan $\D$ with $\D(1)=\{\tau_1,\ldots,\tau_r\}$, $v_i=v_{\tau_i}$, $D_i=D_{\tau_i}$. Let us consider the collection $\widetilde{\D}$ of all cones generated by some rays in $\D(1)$. For a cone $\sigma\in\widetilde{\D}$ let us define a monoid
$$
\Gamma(\sigma)=\sum\limits_{\tau_i\in\D(1)\smallsetminus\sigma(1)}\mathbb Z_{\geqslant0}[D_i].
$$
For a $\mathbb T$-orbit $\mathcal O_\sigma$ the monoids $\Gamma(\mathcal O_\sigma)$(see Introduction) and $\Gamma(\sigma)$ coincide. Let us put
$$
\Upsilon(\D)=\{\Gamma(\sigma): \sigma\in\D\}.
$$

\begin{example}
\label{example_3}
Let us find $\Upsilon(\D)$ for every fan of Example~\ref{example_1}. 
\begin{enumerate}
	\item For $\P^2$ there is one monoid $\mathbb Z_{\geqslant0}$, see Figure~\ref{fig_3}(i).
	\item For $\EuScript H_s$ there are two monoids of rank 2, see Figure~\ref{fig_3}(ii).
	\item For $\P(1,1,s)$ there are monoids $\mathbb Z_{\geqslant0}$ and $s\mathbb Z_{\geqslant0}$, see Figure~\ref{fig_3}(iii).
	\item For $\EuScript B_s$ there are four monoids generated by 2 elements, four monoids generated by 3 elements and a monoid generated by all 4 elements $[D_1]$, $[D_2]$, $[D_3]$, $[D_4]$. We picture monoids generated by 2 elements in 3-dimensional space remembering that $2s[D_3-D_1]=0$, see Figure~\ref{fig_3}(iv).
\end{enumerate}
\begin{figure}[htbp]
	\centering
		\includegraphics[width=120mm]{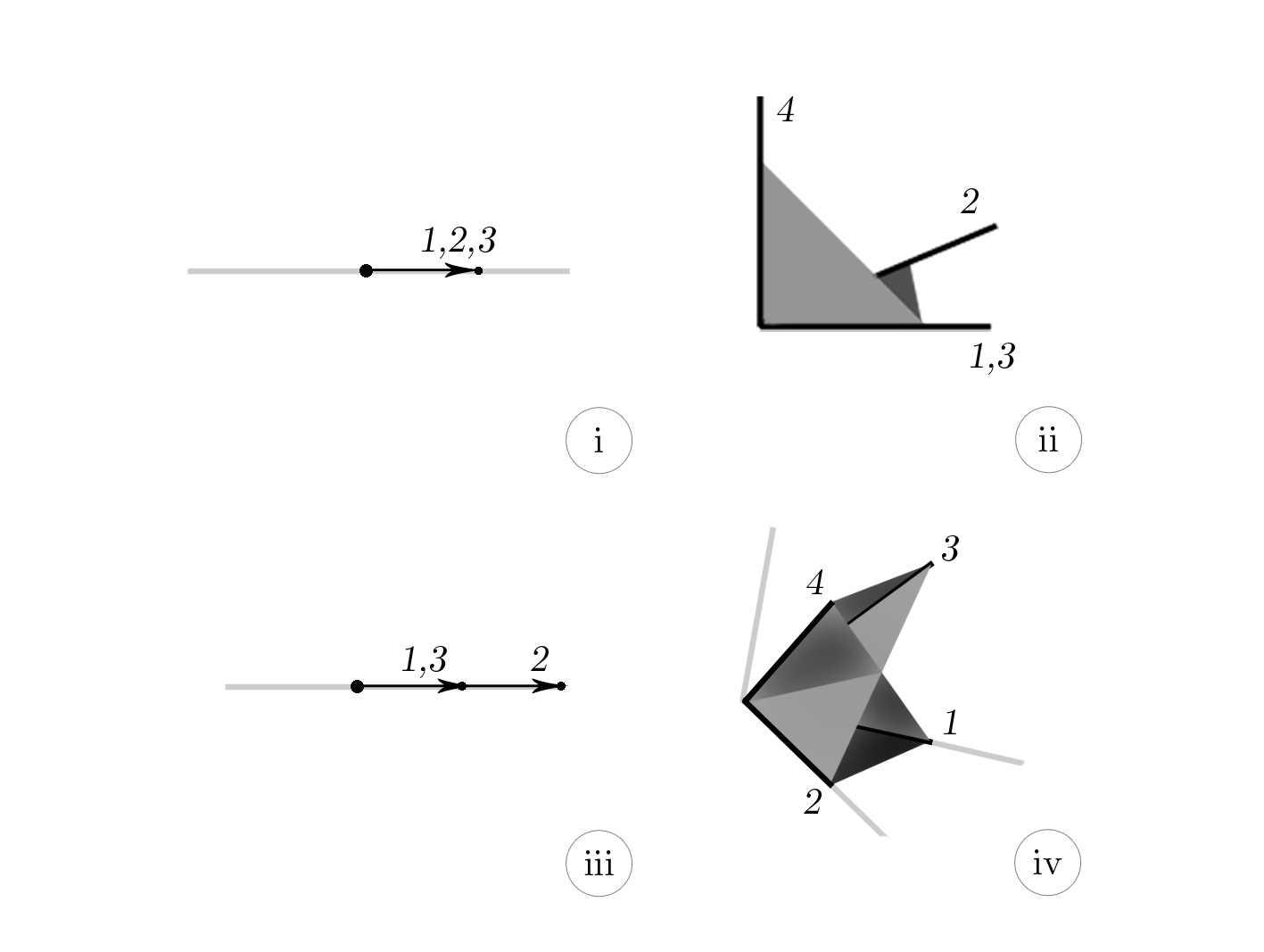}
	\caption{Collections $\Upsilon(\D)$ for $\P^2$, $\EuScript H_s$, $\P(1,1,s)$, and $\EuScript B_s$.}
	\label{fig_3}
\end{figure}
\end{example}

\begin{lemma}
\label{prop_1}
Let $\sigma_1,\ \sigma_2\in\widetilde{\D}$. 
If there exists a root $m\in\R$ such that (\ref{pair}) holds, then
$$
\sigma_1\in\D\ \Leftrightarrow\ \sigma_2\in\D.
$$
\end{lemma}
\begin{proof}
Let us assume that $\sigma_1\in\D$. There exists a $\mathbb T$-orbit $\mathcal O_{\sigma_1}$ on $X$. By Lemmas \ref{H_mconn} and \ref{H_m_lemma}, there exists an orbit $\mathcal O_2$ which meets $H_m\cdot \mathcal O_{\sigma_1}$. By conditions (\ref{pair}), the cone $\sigma_1$, and the root $m$ define $\sigma_2$ uniquely and the cone corresponding to $\mathcal O_2$ is $\sigma_2$. Hence, $\sigma_2\in\D$.

Conversely, by (\ref{pair}), $\sigma_1$ is a face of $\sigma_2$ and $\sigma_1\in\D$.
\end{proof}

\begin{lemma}
\label{prop_2}
Let $\sigma_1,\ \sigma_2\in\widetilde{\D}$ with $\sigma_2(1)=\sigma_1(1)\cup\{\tau_k\}$ for some $1\leq k\leq r$. 
Then conditions (\ref{pair}) hold for some root $m\in\R$ if and only if
$$
\Gamma(\sigma_1)=\Gamma(\sigma_2).
$$
\end{lemma}

\begin{proof}
We have
$$
\Gamma(\sigma_1)=\sum\limits_{\tau_i\notin\sigma_1(1)}{\mathbb Z_{\geqslant0}}[D_i]={\mathbb Z_{\geqslant0}}[D_k]+\sum\limits_{\tau_i\notin\sigma_2(1)}{\mathbb Z_{\geqslant0}}[D_{i}],
$$
$$
\Gamma(\sigma_2)=\sum\limits_{\tau_i\notin\sigma_2(1)}{\mathbb Z_{\geqslant0}}[D_i],
$$
and the condition $\Gamma(\sigma_1)=\Gamma(\sigma_2)$ can be written as
\begin{equation}
\label{eq_d_k}
[D_k]=\sum\limits_{\tau_i\notin\sigma_2(1)}a_i[D_i]
\end{equation}
for some $a_i\in {\mathbb Z_{\geqslant0}}$, $\tau_i\notin\sigma_2(1)$. Let us put $a_k=-1$ and $a_i=0$ for $\tau_i\in\sigma_1(1)$. Due to exact sequence (\ref{eq_exact}), condition (\ref{eq_d_k}) is equivalent to existence of $m\in M$ such that $\pairing{m}{v_i}=a_i$ for all $1\leqslant i\leqslant r$. It means that $m$ is a root and (\ref{pair}) is satisfied for $\sigma_1$, $\sigma_2$.
\end{proof}

The following lemma shows how we can reconstruct the fan $\D$\linebreak from~$\Upsilon(\D)$ and $\D(1)$. 

\begin{lemma}
\label{lemma_oto}
Let $\sigma$ be in $\widetilde{\D}$. Then $\Gamma(\sigma)\in\Upsilon(\D)$ if and only if $\sigma\in\D$.
\end{lemma}
\begin{proof}
The implication $\sigma\in\D\Rightarrow\Gamma(\sigma)\in\Upsilon(\D)$ is trivial.

Let us assume that $\Gamma(\sigma)\in\Upsilon(\D)$. According to the definition of $\Upsilon(\D)$ there exists a cone $\sigma'\in\D$ such that $\Gamma(\sigma')=\Gamma(\sigma)$. Without loss of generality we may assume $\sigma(1)=\{\tau_1,\ldots,\tau_p\}$,  $\sigma'(1)=\{\tau_s,\tau_{s+1}\ldots,\tau_{p+q}\}$. If $\sigma(1)\cap\sigma'(1)\neq\emptyset$, then $\sigma(1)\cap\sigma'(1)=\{\tau_s,\tau_{s+1},$ $\ldots,\tau_p\}$, otherwise we may assume $s=p+1$. Consider a sequence of cones in $\widetilde{\D}$
\begin{equation}
\label{eq_seq}
\sigma=\varsigma_1, \varsigma_2,\ldots,\varsigma_{s+q}=\sigma',
\end{equation}
$$
\varsigma_i=\cone(\tau_i,\ldots,\tau_p),\ 1\leqslant i\leqslant s-1,
$$
$$
\varsigma_s=\cone(\tau_s,\ldots,\tau_p)\mbox{ if }s\leqslant p,\ \varsigma_s=\cone(0)\mbox{ otherwise},
$$
$$
\varsigma_{s+i}=\cone(\tau_s,\ldots,\tau_{p+i}),\ 1\leqslant i\leqslant q.
$$
We have
$$
\Gamma(\varsigma_s)=\Gamma(\sigma)+\Gamma(\sigma')=\Gamma(\sigma)=\Gamma(\sigma').
$$
Moreover, 
$$
\Gamma(\sigma)\subset\Gamma(\varsigma_i)\subset\Gamma(\varsigma_s),\mbox{ hence, } \Gamma(\varsigma_i)=\Gamma(\sigma)\mbox{ for } 1\leqslant i\leqslant s,  
$$
$$
\Gamma(\sigma')\subset\Gamma(\varsigma_{s+i})\subset\Gamma(\varsigma_s),\mbox{ hence, } \Gamma(\varsigma_{s+i})=\Gamma(\sigma')\mbox{ for } 1\leqslant i\leqslant q.  
$$
By Lemma~\ref{prop_2}, for $\varsigma_i, \varsigma_{i+1}$ and some roots $m_i\in M$ conditions (\ref{pair}) are satisfied for all $1\leqslant i\leqslant s+q-1$. By Lemma~\ref{prop_1}, $\varsigma_i\in\D$ for $s+q-1\geqslant i\geqslant 1$. Hence, $\sigma\in\D$.
\end{proof}

\begin{proof}[Proof of Theorem~\ref{Uno}]
Since $\Aut^0(X)$ is generated by $\mathbb T$ and $H_m$, $m\in\R$, orbits $\mathcal O_{\sigma}$ and $\mathcal O_{\sigma'}$ lie in the same $\Aut^0(X)$-orbit if and only if there exists a sequence 
\begin{equation}
\label{eq_seq_2}
\sigma=\varsigma_1, \varsigma_2,\ldots,\varsigma_l=\sigma',
\end{equation}
of cones from $\D$ such that ($\mathcal O_{\varsigma_i}$, $\mathcal O_{\varsigma_{i+1}}$) is $H_{m_i}$-connected for some $m_i\in\R$, $1\leqslant i\leqslant l-1$. 

If $\Gamma(\sigma)=\Gamma(\sigma')$ we can take sequence (\ref{eq_seq}) as the required one. 

Conversely, let us assume there is sequence (\ref{eq_seq_2}). By Lemma~\ref{prop_2}, we have $\Gamma(\sigma)=\Gamma(\varsigma_1)=\ldots=\Gamma(\varsigma_l)=\Gamma(\sigma')$ and $\Gamma(\sigma)=\Gamma(\sigma')$.
\end{proof}

\begin{example}
\label{example_4}
For every toric variety $X$ an $\Aut^0(X)$-orbit on $X$ consists of $\mathbb T$-orbits and can be pictured as a set of cones. We describe these orbits for the varieties from previous examples. 
\begin{enumerate}
	\item The action of $\Aut^0(\P^2)$ on $\P^2$ is transitive, see Figure~\ref{fig_4}(i). 
	\item There are two orbits on~$\EuScript H_s$, see Figure~\ref{fig_4}(ii).
	\item There are two orbits on $\P(1,1,s)$: the regular locus and the singular point, see Figure~\ref{fig_4}(iii).
	\item Any orbit of $\Aut^0(\EuScript B_s)$ on $\EuScript B_s$ is a $\mathbb T$-orbit.
\end{enumerate}
\begin{figure}[htnp]
	\centering
		\includegraphics[width=120mm]{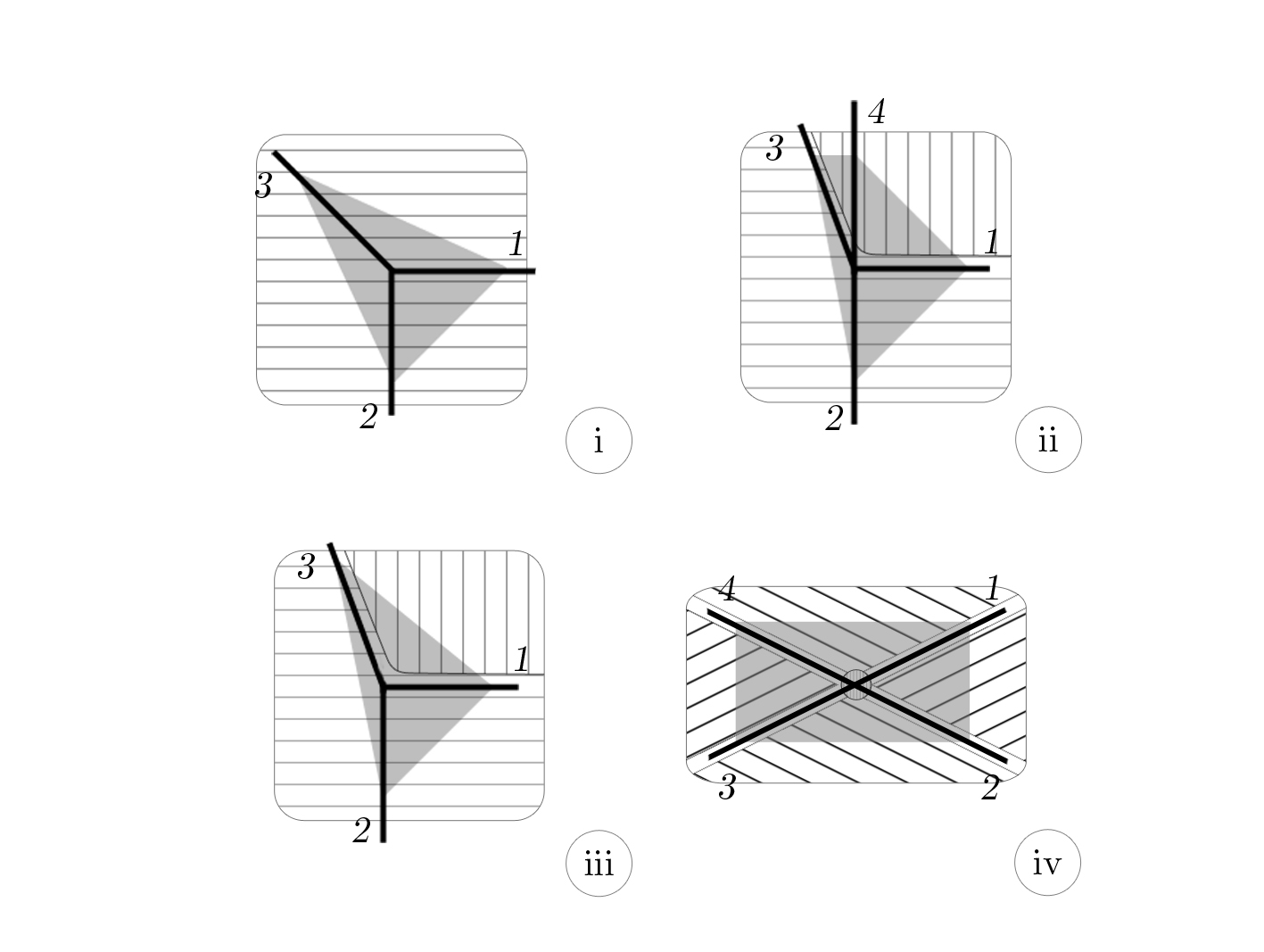}
	\caption{Orbits on $\P^2$, $\EuScript H_s$, $\P(1,1,s)$, and $\EuScript B_s$.}
	\label{fig_4}
\end{figure}
\end{example}

\begin{theorem}
\label{New_Th}
Let $\mathcal O$ and $\mathcal O'$ be $\mathbb T$-orbits. Then
$$
\overline{\Aut^0(X)\cdot\mathcal O}\subset\overline{\Aut^0(X)\cdot\mathcal O'}
\Leftrightarrow \Gamma(\mathcal O)\subset\Gamma(\mathcal O').
$$
\end{theorem}

\begin{proof}
Let us put $\sigma_{\max}=\cone(v_i: [D_i]\notin\Gamma(\mathcal O))$. If $\mathcal O$ corresponds to a cone $\sigma\in\D$, then $\sigma_{\max}\subseteq\sigma$. Hence, $\sigma_{\max}(1)=\sigma_{\max}\cap\D(1)$ and $\Gamma(\sigma_{\max})=\Gamma(\mathcal O)$. By Lemma \ref{lemma_oto}, there exists a $\mathbb T$-orbit $\mathcal O_{\max}$ corresponding to $\sigma_{\max}$. It is easy to see that $\sigma_{\max}\prec\sigma$ and $\mathcal O\subseteq\overline{\mathcal O_{\max}}$. 

Since $\mathcal O$ is an arbitrary $\mathbb T$-orbit with $\Gamma(\mathcal O)=\Gamma(\mathcal O_{\max})$, we have
$$
\mathcal O_{\max}\subset \Aut^0(X)\cdot \mathcal O \mbox{\hspace{0.5cm} and \hspace{0.5cm}} \overline{\mathcal O_{\max}}=\overline{\Aut^0(X)\cdot\mathcal O}.
$$

Without loss of generality we may assume $\overline{\mathcal O}=\overline{\Aut^0(X)\cdot\mathcal O}$ and $\overline{\mathcal O'}=\overline{\Aut^0(X)\cdot\mathcal O'}$. It means that  
$$\sigma(1)=\{v_i: [D_i]\notin\Gamma(\mathcal O)\} \mbox{\hspace{0.5cm} and \hspace{0.5cm}} \sigma'(1)=\{v_i: [D_i]\notin\Gamma(\mathcal O')\}.$$ 
We have
\begin{equation*}
\begin{split}
\Gamma(\mathcal O)\subseteq\Gamma(\mathcal O')&\Leftrightarrow \D(1)\smallsetminus\sigma(1)\subseteq\D(1)\smallsetminus\sigma'(1)\\
&\Leftrightarrow\sigma'\prec\sigma\\
&\Leftrightarrow \overline{\mathcal O}\subseteq\overline{\mathcal O'}\\
&\Leftrightarrow\overline{\Aut^0(X)\cdot\mathcal O}\subseteq\overline{\Aut^0(X)\cdot\mathcal O'}.
\end{split}
\end{equation*}
\end{proof}

\begin{theorem}
\label{Uno_ii} 
Torus orbits $\mathcal O_{\sigma}$ and $\mathcal O_{\sigma'}$ on $X$ lie in the same $\Aut(X)$-orbit if and only if there exists an automorphism $\phi:\Cl(X)\to\Cl(X)$ with the following properties
	 	\begin{itemize}
			\item $\phi(\Gamma(\mathcal O_{\sigma}))=\Gamma(\mathcal O_{\sigma'})$,
			\item $\phi(\Upsilon(\D))=\Upsilon(\D)$,
	 	  \item there exists a permutation $f$ of elements in $\{1,\ldots,r\}$ such that $\phi([D_i])=[D_{f(i)}]$.		
		\end{itemize}
\end{theorem}

\begin{proof}
Let us assume $\mathcal O_{\sigma}$ and $\mathcal O_{\sigma'}$ lie in the same $\Aut(X)$-orbit. By Theorem \ref{all_aut}, there exists $P_\psi\in\Aut(X)$ defined by some $\psi\in\Aut(N,\D)$ that maps $\Aut^0(X)\cdot\mathcal O_{\sigma}$ to $\Aut^0(X)\cdot\mathcal O_{\sigma'}$. The automorphism $P_\psi$ defines an automorphism $\phi:\Cl(X)\to\Cl(X)$ with the following properties:
\begin{itemize}
	 	  \item $\phi([D_\tau])=[P_\psi(D_{\tau})]=[D_{\psi(\tau)}]$;
			\item $\phi(\Gamma(\sigma_1))=\Gamma(\sigma_2)$;
			\item $\phi(\Upsilon(\D))=\Upsilon(\D)$.
\end{itemize}

Conversely, for every $\phi$ let us find $P_\psi\in\Aut(X)$ mapping \linebreak $\Aut^0(X)$-orbit corresponding to $\Gamma(\sigma)$ to $\Aut^0(X)$-orbit corresponding to $\Gamma(\sigma')=\phi(\Gamma(\sigma))$. There is a commutative diagram: 
$$
\xymatrix{
0 \ar[r]& M \ar[r]\ar@{.>}[d]^{\psi^*}& \mathrm{Div}_{\mathbb T}(X) \ar[r]\ar[d]^{\tilde{f}}& \Cl(X) \ar[r]\ar[d]^{\phi}& 0\\
0 \ar[r]& M \ar[r]                    & \mathrm{Div}_{\mathbb T}(X) \ar[r]                  & \Cl(X) \ar[r]& 0
}
$$
where $\tilde{f}$ is defined by $f$. There exists an automorphism $\psi^*:M\to M$ making the diagram commutative. Let us define $\psi:N\to N$ as dual to $\psi^*$. One can check that $\psi$ preserves $\D(1)$. By Lemma~\ref{lemma_oto}, the cone $\psi(\varsigma)$, $\varsigma\in\D$, belongs to $\D$, since $\Gamma(\psi(\varsigma))=\phi(\Gamma(\varsigma))\in\Upsilon(\D)$. The automorphism $\psi$ defines the desired $P_\psi$.
\end{proof}

\begin{example}
\label{example_5}
It is easy to see that $\Aut^0$- and $\Aut$-orbits on $\P^2$, $\EuScript H_s$, and $\P(1,1,s)$ coincide.

Let us describe the $\Aut(\EuScript B_s)$-orbits on $\EuScript B_s$. For $s>1$ the group $\mathbb Z_2\oplus\mathbb Z_2$ acts on $\Upsilon(\D)$ and defines orbits on $\EuScript B_s$ in the following way: the open $\Aut(\EuScript B_s)$-orbit coincides with the open $\mathbb T$-orbit, one dimensional torus orbits belong to the same $\Aut(\EuScript B_s)$-orbit, and four $\mathbb T$-stable points form two $\Aut(\EuScript B_s)$-orbits, see Figure~\ref{fig_4}(iv). For $s=1$ all $\mathbb T$-fixed points on $\EuScript B_1$ are in the same $\Aut(\EuScript B_1)$-orbit. 
\end{example}

\begin{theorem}
\label{Dos}
Let $X$ be a complete toric variety. The group $\Aut(X)$ acts on $X$ transitively if and only if 
	$$
	X\cong\prod\limits_{i=1}^{s}\P^{k_i}
	$$ 
	for some $k_i\in\mathbb N$.
\end{theorem}

\begin{proof}
Let us assume that the action of $\Aut(X)$ is transitive. Then $X$ is smooth, in particular, all cones of $\D$ are simplicial. It is easy to see that open $\Aut^0(X)$- and open $\Aut(X)$-orbits coincide. Hence, $\Upsilon(\D)$ consists of one monoid.

Let $\sigma\in\D$ be a cone of maximal dimension $d$. We may assume that $\sigma(1)=\{\tau_1,\ldots,\tau_d\}$. There is a cone  $\sigma'\in\D$ of dimension $d$ which meets $\sigma$ along the facet without~$\tau_1$: $\sigma'(1)=\{\tau_2,\ldots,\tau_d,\tau_{d+1}\}$.

Since $\Upsilon(\D)$ contains only one monoid, the following monoids coincide:
$$
\Gamma(\sigma')=\mathbb Z_{\geqslant0}[D_1]+\sum\limits_{i>d+1}\mathbb Z_{\geqslant0}[D_i],
$$
$$
\Gamma(\sigma)=\mathbb Z_{\geqslant0}[D_{d+1}]+\sum\limits_{i>d+1}\mathbb Z_{\geqslant0}[D_i].
$$
We have
\begin{equation}
\label{DD}
\begin{split}
[D_1]=a_{d+1}[D_{d+1}]+\sum\limits_{i>d+1}a_i[D_i],\\
[D_{d+1}]=b_{1}[D_{1}]+\sum\limits_{i>d+1}b_i[D_i],
\end{split}
\end{equation}
where $a_{d+1},b_1\in\mathbb Z_{\geqslant0}$ and $a_i,b_i\in\mathbb Z_{\geqslant0}$ for $d+2\leq i\leq r$. Let us put $a_1=-1$,  $b_{d+1}=-1$ and $a_i=0$, $b_i=0$ for $2\leq i\leq d$.

Due to equations (\ref{DD}) there exist $m_1, m'_1\in M$ such that 
$$
\pairing{m_1}{v_i}=a_i\mbox{ for }1\leq i\leq r,
$$
$$
\pairing{m'_1}{v_i}=b_i\mbox{ for }1\leq i\leq r.
$$
Since $\pairing{m_1}{v_i}=\pairing{m'_1}{v_i}=0$, $2\leq i\leq d$, the characters $m_1$ and $m'_1$ are proportional. Hence, $a_i=-b_i$ for all $i$? in particular, $a_i=b_i=0$ for $i>d+1$. We have $[D_1]=[D_{d+1}]$ and $\pairing{m_1}{v_1}=-1$, $\pairing{m_1}{v_{d+1}}=1$, $\pairing{m_1}{v_i}=0$, $i\neq 1,d+1$. It means that $m_1\in\R$ is a semisimple root.

Since we can take any element of $\sigma(1)$ instead of $\tau_1$ there exists a set of semisimple roots $m_1,\ldots,m_d$ such that $\pairing{m_i}{v_j}=-\delta_{ij}$ for all $1\leqslant i,j\leqslant d$. Hence, $m_1,\ldots,m_d$ are linearly independent.

In \cite[Theorem 3.18]{Nill} it is shown that if $X$ is a complete toric variety of dimension $d$ and there are $d$ linearly independent semisimple roots $m_1,\ldots,m_d$, then 
	$$
	X\cong\prod\limits_{i=1}^{s}\P^{k_i}
	$$ 
for some $k_i\in\mathbb N$. Applying this theorem completes the proof in one direction.

Conversely, the group 
	$$
\prod\limits_{i=1}^{s}\mathrm{PGL}(k_i+1),
	$$ 
acts on $\prod\limits_{i=1}^{s}\P^{k_i}$ transitively.
\end{proof}


\begin{thebibliography}{99}
	\bibitem{AG}
	I.\,Arzhantsev, S.\,Gaifullin: {\it Homogeneous toric varieties.} J. Lie Theory {\bf20} (2010), 283-293.
	\bibitem{AKZ} 
	I.\,Arzhantsev, K.\,Kuyumzhiyan, M.\,Zaidenberg: {\it Flag varieties, toric varieties, and suspensions: three instances of infinite transitivity.} \href{http://www.arxiv.org/abs/1003.3164v1}{{\tt arXiv:1003.3164v1}} (2010), 25~p.
	\bibitem{BH}
	F.\,Berchtold, J.\,Hausen: {\it Bunches of cones in the divisor class group --- a new combinatorial language for toric varieties.} Int. Math. Res. Not. {\bf 6} (2004), 261-302.
	\bibitem{BG}
	W.\,Bruns, J.\,Gubeladze: {\it Polytopal linear groups.} J. Algebra {\bf 218} (1999), 715-737. 
	\bibitem{Buh}
	D.\,B\"uhler: {\it Homogener Koordinatenring und Automorphismengruppe vollst\"andiger torischer
Variet\"aten.} Diplomarbeit, University of Basel, 1996.
	\bibitem{Cox2}
	D.\,Cox: {\it The homogeneous coordinate ring of a toric variety.} J. Alg. Geometry {\bf 4} (1995), 17-50.
	\bibitem{Cox1}
	D.\,Cox, J.\,Little, H.\,Shenck: {\it Toric varieties.} Graduate Studies in Math. {\bf 124}, AMS, Providence, RI, 2011.
	\bibitem{Dem1}
	M.\,Demazure: {\it Sous-groupes algebriques de rang maximum du groupe de Cremona.} Ann. Sci. \'Ecole Norm. Sup. {\bf 3} (1970) 507-588.
	\bibitem{Fulton}
	W. Fulton: {\it Introduction to toric varieties.} Ann. of Math. Studies {\bf 131}, Princeton University Press, Princeton, NJ, 1993.	
	\bibitem{Nill}
	B.\,Nill: {\it Complete toric varieties with reductive automorphism group.} Math. Z. {\bf 252} (2006), 767-786.
	\bibitem{Oda1}
	T.\,Oda: {\it Convex bodies and algebraic geometry --- An introduction to the theory of toric varieties.} Springer Verlag, Berlin, 1988.
	\end{thebibliography}
\end{document}